\documentclass[12pt]{article}
\usepackage{amssymb,amsmath,amsthm,tikz,multirow,amscd}
\usetikzlibrary{arrows,calc}

\title{Fixed Point Sets and the Fundamental Group II:  Euler Characteristics}
\author{Sylvain Cappell \\
Courant Institute of Mathematical Sciences, New York University \\
\\
Shmuel Weinberger\thanks{Research was partially supported by NSF grants DMS 0504721 and 0805913} \\
University of Chicago \\
\\
Min Yan\thanks{Research was supported by Hong Kong RGC General Research Fund 16319116.} \\ Hong Kong University of Science and Technology}

\newcommand{\sub}{\subset}
\newcommand{\pa}{\partial}

\newcommand{\mc}{\mathcal}
\newcommand{\bb}{\mathbb}

\newtheorem{theorem}{Theorem}
\newtheorem*{theorem*}{Theorem}
\newtheorem{lemma}[theorem]{Lemma}

\newtheorem*{corollary*}{Corollary}

\theoremstyle{definition}
\newtheorem*{definition*}{Definition}
\newtheorem*{case*}{Case}
\newtheorem*{subcase*}{Subcase}
\newtheorem{remark}{Remark}
\newtheorem{example}{Example}

\numberwithin{equation}{section}

\begin{document}

\maketitle

\begin{abstract}
For a group $G$ of not prime power order, Oliver showed that the obstruction for a finite CW-complex $F$ to be the fixed point set of a contractible finite $G$-CW-complex is the Euler characteristic $\chi(F)$. He also has the similar results for compact Lie group actions. We show that the analogous problem for $F$ to be the fixed point set of a finite $G$-CW-complex of some given homotopy type is still determined by the Euler characteristic. Using trace maps in $K_0$, we also see that there are interesting roles for the fundamental group and the component structure of the fixed point set.
\end{abstract}

\section{Introduction}

The classical problem of what are the possible homotopy types of fixed sets of group actions on contractible complexes was solved by L. Jones \cite{jones} in the case of $p$-groups, and R. Oliver \cite{oliver} for non-$p$-groups. Our aim in this paper is to put the work of Oliver into a more general context, exposing where the fundamental group does and where it does not play a role. The analogous generalization of Jones' work has very different features, which we will discuss in \cite{cwy-jones}.

The following is our original motivation. Suppose $G$ is a finite group acting by isometries on a compact Riemannian manifold $X$ with nonpositive curvature, and the center of the fundamental group $\pi_1X$ is trivial. Then the $G$-action has a fixed point if and only if the induced homomorphism $G\to \text{Out}(\pi_1X)$ lifts to $\text{Aut}(\pi_1X)$. Here $\text{Out}(\pi_1X)$ is the group of outer automorphisms, and $\text{Aut}(\pi_1X)$ is the group of genuine automorphisms. By taking the fundamental group to be based at a fixed point, we get the necessity of the condition. The sufficiency follows by considering the group $\Gamma$ of lifts of $G$ to the universal cover. Since the center of $\pi_1X$ is trivial, the map $\Gamma\to G$ splits. By Cartan's fixed point theorem \cite{bh}, any finite group of isometries of a simply connected non-positively curved manifold has a fixed point.

If the action is not required to be by isometries, would the action still have a fixed point? For a $p$-group acting on a finite $K(\pi,1)$ complex with centerless $\pi$, Smith theory says that any action satisfying the lifting condition must have a fixed point. 

What about actions by a non-$p$-group?

\begin{theorem}\label{aspherical}
For any finite group $G$ of not prime power order, there is a $G$-action on a compact aspherical manifold $X$, with centerless $\pi=\pi_1X$, such that the induced homomorphism $G\to \text{\rm Out}(\pi_1X)$ lifts to $\text{\rm Aut}(\pi_1X)$, and yet the action has no fixed point.
\end{theorem}

This result uses the method of Davis \cite{davis}, which promotes constructions involving finite complexes and compact aspherical manifolds with boundary to closed aspherical examples. Consequently, we were led to investigate the role of the fundamental group in generalizing the theory of Oliver. In \cite{cwy-jones}, we will see a contrastingly different role in generalizing the theory of Jones.

It would be reasonable to believe that $\pi_1X$ plays a role in understanding the possible fixed sets of group actions on $X$. In both \cite{jones} and \cite{oliver}, at key points in the construction, an obstruction in $K_0({\bb Z}G)$ arises, and one would expect that for non-simply connected $X$, an analogous element of $K_0({\bb Z}\Gamma)$ arises.
However, in this paper we shall see that the fundamental group does not intervene in this way for non-$p$-groups (though, as we shall see, it does affect the possible trade-offs between Euler characteristics of different components of fixed points). On the other hand, we shall see in \cite{cwy-jones} that this is the case for $p$-groups.

We begin with some definitions. We say that a $G$-map $f\colon X \to Y$ between two {\em finite} $G$-complexes is a {\em pseudo-equivalence} if it is an unequivariant homotopy equivalence. The pseudo-equivalence property is equivalent to that the the induced map $X\times EG\to Y\times EG$ on the Borel constructions is a $G$-homotopy equivalence. Note that $X\times EG\to X$ is a $G$-map and an unequivariant homotopy equivalence. However, the Borel construction is usually not a finite complex, so that the map is not a pseudo-equivalence in the sense of the present paper. 

A pseudo-equivalent $G$-map does not necessarily have an inverse pseudo-equivalent $G$-map. To make pseudo-equivalence into an equivalence relation, therefore, we need to allow two finite $G$-complexes $X,Y$ to be pseudo-equivalent if they are related by a zig-zag sequence of pseudo-equivalent $G$-maps (all $Z_i,W_i$ are finite $G$-complexes)
\[
X\xleftarrow{f_1}
Z_1 \xrightarrow{g_1}
W_1 \xleftarrow{f_2}
Z_2 \xrightarrow{g_2}
W_2 \leftarrow \cdots \to
W_{n-1}\xleftarrow{f_n}
Z_n \xrightarrow{g_n}
Y.
\]
Then for any $p$-subgroup $P$ of $G$, the Smith theory can be applied to the fixed sets of $P$ to give $H_*(X^P;{\bb F}_p)\cong H_*(Y^P;{\bb F}_p)$. In particular, this implies that, if $G$ acts freely on $X$, then the $G$-action on $Y$ must also be free. 

If $G$ is a non-$p$-group, however, it is still possible for $X$ to have no $G$-fixed point, and for $Y$ to have $G$-fixed point. For the case $Y$ is a single point, i.e., $X$ is contractible, Oliver proved the following.

\begin{theorem*}[Oliver]
For any finite group $G$ of not prime power order, there is a number $n_G$, such that a finite complex $F$ is (homotopy equivalent to) the fixed set on a contractible complex, if and only if $\chi(F)=1$ mod $n_G$.
\end{theorem*}

By \cite[Theorem 5]{oliver} and the subsequent corollary, we know $n_G=1$, i.e., there is no condition at all on $F$, if and only if $G$ is {\em not} of the form $P\lhd H\lhd G$, with $P$ and $G/H$ having prime power orders, and $H/P$ cyclic. For example, for $G$ a non-solvable group or an abelian group with at least three non-cyclic Sylow subgroups, we have $n_G=1$.

We also know $n_G=0$, i.e., the obstruction is exactly the Euler characteristic, if and only if $G$ has a normal subgroup $P$ of prime power order, such that $G/P$ is cyclic. In particular, a cyclic group has $n_G=0$.

For the complete determination of $n_G$, see \cite{oliver3, oliver4}.

The following is our extension of Oliver's theorem to general $Y$.

\begin{theorem}\label{connected}
Suppose $G$ is a group of not prime power order, and $Y$ is a finite $G$-complex with non-empty and connected fixed set $Y^G$. Then $F$ is the fixed set of a $G$-action on a finite $G$-complex pseudo-equivalent to $Y$ if and only if $\chi(F)=\chi(Y^G)$ {\rm mod} $n_G$.
\end{theorem}

\begin{remark}
Given the Euler characteristic condition, we actually construct a finite $G$-complex $X$ and a pseudo-equivalent $G$-map $X\to Y$, such that $F=X^G$. In particular, there is no need to use a zig-zag sequence of maps for the pseudo-equivalence between $X$ and $Y$. 
\end{remark}

In Oliver's work, there is another number $m_G$ which is the largest square-free factor of $n_G$. As shown in \cite[Section 2]{quinn}, this arises in answering the question about $G$-actions which are $G$-ANRs (equivariant absolute neighborhood retracts) but not necessarily homotopy finite\footnote{Recall that for $G$ trivial, according to West's celebrated theorem \cite{west}, any finite dimensional ANR is homotopy equivalent to a finite complex. When $G$ is nontrivial, as Quinn's examples show, this is not true. Moreover, elementary examples show that there is no analogue of Oliver's theorem for general topological actions.}. The same method, combined with the arguments of the present paper, proves the $G$-ANR analogues of all our results, with $m_G$ replacing $n_G$. Such results are relevant to locally linear group actions on topological manifolds. The following is the $G$-ANR analogue of Theorem \ref{connected}. Of course, the statement requires us to slightly extend pseudo-equivalence to the compact $G$-ANR category.

\begin{theorem}\label{connected_anr}
Suppose $G$ is a group of not prime power order, and $Y$ is a compact $G$-ANR with non-empty and connected fixed set $Y^G$. Then a finite complex $F$ is the fixed set of a $G$-action on a compact $G$-ANR pseudo-equivalent to $Y$ if and only if $\chi(F)=\chi(Y^G)$ {\rm mod} $m_G$.
\end{theorem}

We say that an action on $X$ is {\em pseudo-trivial}, if there is a pseudo-eqivalence to $X$ with trivial action. This is tantamount to saying that the map $X\to X/G$ splits.  Any action on a contractible complex is pseudo-trivial. 

If we apply Theorem \ref{connected} to the case the action on $Y$ is trivial, we find that $F$ is the fixed point of a pseudo-trivial $G$-complex homotopic to $Y$ if and only if $\chi(F)=\chi(Y)$ mod $n_G$.

The fundamental group does not appear in Theorem \ref{connected}. However, there is an important role for the fundamental group when $Y^G$ is not connected. The fundamental group and the action group interact as follows. As before, let $\Gamma$ be the group of liftings of $G$-actions on $Y$ to the universal cover $\tilde{Y}$. We have an exact sequence
\[
1\to \pi_1 Y \to \Gamma\to G\to 1.
\]
If $\tilde{y}\in \tilde{Y}$ covers $y\in Y$, then we have isomorphism of isotropy groups $\Gamma_{\tilde{y}}\cong G_y$. For $y\in Y^G$, this gives a splitting of the exact sequence. Various choices of $\tilde{y}$ over $y$ give a conjugacy class of splittings, and the conjugacy class depends only on the connected component of $Y^G$ containing $y$. 

To see the interaction of $G$ with $\pi_1Y$ at the level of fixed sets, we introduce the following definition.

\begin{definition*}
A $G$-action on $Y$ is {\em weakly $G$-connected} if each component of $Y^G$ corresponds to a different conjugacy class of lifts of $G$ to $\Gamma$. 
\end{definition*}

Recall that $G$-space $Y$ is said to be {\em $G$-connected}, if $Y^H$ is non-empty and connected for every subgroup $H$ of $G$. Evidently, $G$-connected implies weakly $G$-connected, but the converse is often false. For example, a free action is weakly $G$-connected but not $G$-connected. Another example is actions by isometries on non-positively curved spaces, again by Cartan's fixed point theorem. 

\begin{theorem}\label{component_wise}
Suppose $G$ is a cyclic group or satisfies $n_G=0$, $F$ is a finite complex, and $Y$ is a finite $G$-complex. If $Y$ is weakly $G$-connected, then a $G$-map $f\colon F\to Y$ can be extended to a pseudo-equivalence $g\colon X\to Y$, with finite $G$-complex $X$ satisfying $F=X^G$, if and only if for each connected component $C$ of $Y^G$, we have $\chi(f^{-1}(C))=\chi(C)$ mod $n_G$. 
\end{theorem}

Since connected $Y^G$ implies weakly $G$-connected, Theorem \ref{component_wise} generalizes Theorem \ref{connected} in case $n_G=0$. 

The theorem allows $f^{-1}(C)=\emptyset$ for some components. On the other hand, we do allow $f^{-1}(C)$ to be disconnected for some component. Then $X$ will no longer be weakly $G$-connected, but our weakly $G$-connected hypothesis is on $Y$. Our proof of the theorem uses the Hattori-Stallings trace \cite{hattori}, perhaps the most classical of the trace maps from $K$-theory to Hochschild (or even cyclic) homology.

Even without the weakly $G$-connected assumption, our methods apply to give some new constructions of actions, and some new necessary congruence conditions. In this general setting, our results can be profitably compared to earlier work of Oliver and Petrie \cite{op} (which deals with a somewhat different problem) and Morimoto and Iizuka \cite{mi} (which is about pseudo-equivalence of $G$-complexes with finite fundamental groups). 

\begin{theorem}\label{obstruct}
Suppose $G$ is a finite group of not prime power order, and $Y$ is a finite $G$-complex.  Then there is a subgroup $N_Y$ of ${\bb Z}^{\pi_0Y^G}$, such that a $G$-map $f\colon F\to Y$ can be extended to a pseudo-equivalence $g\colon X\to Y$, with finite $G$-complex $X$ satisfying $F=X^G$, if and only if 
\[
(\chi(f^{-1}(C))-\chi(C))\in N_Y.
\]
\end{theorem}

For $f\colon F\to Y$ to be extended to a pseudo-equivalence $g\colon X\to Y$, it is necessary that the {\em global} Euler characteristic condition $\chi(F)=\chi(Y^G)$ mod $n_G$ is satisfied. The reason is that the mapping cone $cX\cup_g Y$ of the homotopy equivalence $g$ is a contractible $G$-CW-complex with fixed set $(cX\cup_g Y)^G=cF\cup Y^G$. Then Oliver's theorem implies $\chi(cF\cup Y^G)=1+\chi(Y^G)-\chi(F)=1$ mod $n_G$. 

Conversely, the following result says that, if the {\em local} Euler characteristic condition is satisfied for each connected component of $Y^G$, then we have the pseudo-equivalence extension.

\begin{theorem}\label{local}
Suppose $G$ is a group of not prime power order, $F$ is a finite complex, $Y$ is a connected finite $G$-complex, and $f\colon F\to Y$ is a $G$-map. If $\chi(f^{-1}(C))=\chi(C)$ {\rm mod} $n_G$ for every connected component $C$ of $Y^G$, then $f$ can be extended to a pseudo-equivalence $g\colon X\to Y$, with finite $G$-complex $X$ satisfying $F=X^G$.
\end{theorem}

The discussion on global versus local Euler characteristic conditions means exactly that the obstruction group $N_Y$ introduced in Theorem \ref{local} satisfies
\[
n_G{\bb Z}^A\sub N_Y
\sub \{(a_C)\in {\bb Z}^A\colon n_G\text{ {\rm divides} }\textstyle\sum a_C\},\quad 
A=\pi_0Y^G.
\]
We note two cases in which the two extremes coincide. If $Y^G$ is connected, then $N_Y=n_G{\bb Z}$, and we recover Theorem \ref{connected}. If $n_G=1$, then $N_Y={\bb Z}^A$, which means there is no obstruction at all.
 
Under the condition of Theorem \ref{component_wise}, we know $N_Y$ is the lower bound. To get a somewhat non-trivial example that $N_Y$ is the upper bound, we consider a finite contractible $G$-complex $Y$, and any $G$-map $f\colon F\to Y$ satisfying the global Euler characteristic condition $\chi(F)=\chi(Y^G)=1$ {\rm mod} $n_G$. By Oliver's theory \cite{oliver}, there is a finite contractible $G$-complex $X$, such that $F=X^G$. If we extend $f\colon F\to Y$ to a $G$-map $g\colon X\to Y$, then $g$ is a pseudo-equivalence due to the contractibility of $X$ and $Y$. The extension exists as long as all the fixed sets $Y^H$ are sufficiently highly connected. There is a $G$-representation $V$, such that the representation sphere $S(V)$ has no fixed point by $G$, and is highly connected for the fixed sets of proper subgroups (see the first remark after the proof of Lemma \ref{rep}). Then $Z=Y*S(V)$ is still contractible, and satisfies $Z^G=Y^G$. Moreover, $Z$ becomes highly connected for the fixed sets of proper subgroups. The contractibility implies that the natural map $h\colon Y\to Z$ is a pseudo-equivalence. Therefore $h\circ f\colon F\to Z$ can be extended to a $G$-pseudo-equivalence $g\colon X\to Z$.

In general, we have not yet completely ascertained the pattern which determines where between the two given extremes $N_Y$ actually is, although one can easily produce examples where either extreme is realized. 

\begin{remark}
That being said, if $Y^G=\emptyset$, then it is impossible to construct a pseudo-equivalent action with nonempty fixed set in the context of this paper. However, using zig-zag compositions of pseudo-equivalence maps, sometimes, but not always, it is possible to get from such an action to one whose fixed set is arbitrary.
\end{remark}

Our results about finite $G$-complexes (and compact $G$-ANRs) apply to compact Lie group actions, as well. In \cite{oliver2}, Oliver extended his theorem to the case $G$ is a compact Lie group. He calls $G$ {\em $p$-toral} if the identity component $G_0$ is a torus, and $G/G_0$ is a $p$-group. This is precisely the case Smith theory applies. Then Oliver's (and Quinn's) results apply to his problem in the case when $G$ is not $p$-toral, for any prime $p$. This means two possibilities: (1) If $G/G_0$ is not of prime power order and $G_0$ is a torus, then the results for $G$ are the same as those for $G/G_0$; (2) If $G_0$ is not a torus, then we have $n_G=1$, and there are no Euler characteristic obstruction to making fixed sets. By the direct geometric arguments in Section \ref{nonempty} and \ref{empty}, this also holds in our situation. 

The paper is organized as follows. In Sections \ref{partition} and \ref{nonempty}, we prove Theorem \ref{local} for the case all $F_C=f^{-1}(C)$ are not empty. The case some $F_C=\emptyset$ requires separate treatment, and we prove this case in Section \ref{empty}. This completes the proof of Theorem \ref{local}. Then Theorem \ref{connected} immediately follows from Theorem \ref{local}. The case $F_C=\emptyset$ is significant because it gives actions without fixed points. In fact, we also prove Theorem \ref{aspherical} in Section \ref{empty}. 

In Section \ref{obstruction}, we use a formal construction and Theorem \ref{local} to prove Theorem \ref{obstruct}. In Section \ref{fundgroup}, we prove two results that give necessary conditions on the Euler characteristics. Theorems \ref{cyclic} is for rational pseudo-equivalence under a cyclic group action. Theorem \ref{comp-wise} is for pseudo-equivalence under the action by a group satisfying $n_G=0$. Theorem \ref{component_wise} is a consequence ofthe two theorems.

We would like to thank K. Pawalowsky, R. Oliver and M. Morimoto for valuable conversations about this work.

\section{Cell-wise Partition of Euler Characteristic}
\label{partition}

The proof of Theorem \ref{local} starts with the observation that we can inductively apply Oliver's construction to cells of $Y^G$. Recall that a CW-complex $Y$ is regular if every cell $\sigma$ is given by an embedding $D^k\to Y$. This implies that the boundary $\pa\sigma$ of the cell is a sphere $S^{k-1}$ embedded in $Y$. Then $\chi(\sigma)=\chi(D^{\dim\sigma})=1$, and $\chi(\pa\sigma)=\chi(S^{\dim\sigma-1})=1-(-1)^{\dim\sigma}$.

\begin{lemma}\label{lem1}
Suppose $Y$ is a regular complex, and $f\colon F\to Y$ is a map. If $\chi(f^{-1}(\sigma))=1$ {\rm mod} $n$ for every cell $\sigma$ of $Y$, then $\chi(F)=\chi(Y)$ {\rm mod} $n$.
\end{lemma}

The lemma allows $n=0$, which means dropping ``mod $n$'' in the statement. In the following proof, all numerical equalities are true mod $n$.

\begin{proof}
If $\dim Y=0$, then $Y$ consists of finitely many points $y_1,y_2,\dots,y_k$, and $F=\cup_{i=1}^k f^{-1}(y_i)$ is a disjoint union. By the assumption, we have $\chi(f^{-1}(y_i))=1$. Therefore $\chi(F)=\sum_{i=1}^k \chi(f^{-1}(y_i))=k=\chi(Y)$.

Suppose $\dim Y=d$, and the lemma is proved for complexes of dimension $<d$. We have $Y^{d-1}\sub Y_1\sub\cdots\sub Y_{k-1}\sub Y_k=Y$, where $Y^{d-1}$ is the $(d-1)$-skeleton of $Y$, and $Y_i$ is obtained by attaching one $d$-cell to $Y_{i-1}$. By the inductive assumption, we have $\chi(f^{-1}(Y^{d-1}))=\chi(Y^{d-1})$. Suppose we already proved $\chi(f^{-1}(Y_{i-1}))=\chi(Y_{i-1})$. Let $Y_i=Y_{i-1}\cup \sigma$ for a $d$-cell $\sigma$. Then $Y_{i-1}\cap \sigma$ is a complex of dimension $<d$. Therefore we have $\chi(f^{-1}(Y_{i-1}\cap \sigma))=\chi(Y_{i-1}\cap \sigma)$ by the inductive assumption. Combined with $\chi(f^{-1}(\sigma))=1=\chi(\sigma)$, we get
\begin{align*}
\chi(f^{-1}(Y_i))
&=\chi(f^{-1}(Y_{i-1}))+\chi(f^{-1}(\sigma))-\chi(f^{-1}(Y_{i-1}\cap \sigma)) \\
&=\chi(Y_{i-1})+\chi(\sigma)-\chi(Y_{i-1}\cap \sigma)
=\chi(Y_i).
\end{align*}
Inductively, this proves $\chi(F)=\chi(f^{-1}(Y_k))=\chi(Y_k)=\chi(Y)$.
\end{proof}

The converse of Lemma \ref{lem1} is true up to homotopy equivalence.

\begin{lemma}\label{lem2}
Suppose $Y$ is a connected regular complex. Suppose $F\ne\emptyset$ and $f\colon F\to Y$ is a map, such that $\chi(F)=\chi(Y)$ {\rm mod} $n$. Then there is a homotopy equivalence $\phi\colon F\simeq \hat{F}$ and a map $\hat{f}\colon \hat{F}\to Y$, such that $\hat{f}\phi\simeq f$, and $\chi(\hat{f}^{-1}(\sigma))=1$ {\rm mod} $n$ for every cell $\sigma$ of $Y$.
\end{lemma}

\begin{proof}
We call a cell {\em top cell} if it is not in the boundary of any other cell. Fix $c\in F$ and a top cell $\beta$ containing $f(c)$. We regard $\beta$ as the ``base cell'' of $Y$. For each top cell $\sigma$ different from $\beta$, there is a continuous path $\gamma\colon [-1,1]\to Y$, such that $\gamma(-1)=f(c)$, $\gamma(-1,0)\cap \sigma=\emptyset$, and $\gamma(0,1]\sub\mathring{\sigma}$. This implies that $\gamma(0)\in\pa\sigma$, and $f(c)$ is the only other possible point on the path lying inside $\pa\sigma$.

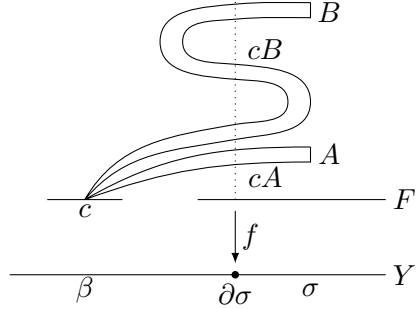
\begin{figure}[htp]
\centering
\begin{tikzpicture}[>=latex]

\draw
	(-3,-1) -- (2,-1)
	(-2.5,0) -- (-1.5,0)
	(-0.5,0) -- (2,0)
	(-2,0) to[out=20,in=180] 
	(1,0.5) -- (1,0.7)
	(-2,0) to[out=30,in=180] 
	(1,0.7)
	(-2,0) to[out=50,in=190] 
	(0,0.8) to[out=10,in=270] 
	(1,1.3) to[out=90,in=-5] 
	(0,1.8) to[out=175,in=270] 
	(-0.7,2.1) to[out=90,in=185] 
	(0,2.4) to[out=5,in=180]
	(1,2.42) -- (1,2.62)
	(-2,0) to[out=60,in=190] 
	(0,1) to[out=10,in=270] 
	(0.7,1.3) to[out=90,in=-5] 
	(0,1.6) to[out=175,in=270] 
	(-1,2.1) to[out=90,in=185] 
	(0,2.6) to[out=5,in=180]
	(1,2.62);

\draw[dotted]
	(0,0) -- (0,2.7);
	
\draw[->]
	(0,-0.15) -- (0,-0.85);

\fill
	(0,-1) circle (0.05);

\node at (1.25,0.6) {\small $A$};
\node at (1.25,2.5) {\small $B$};
\node at (0.4,0.3) {\small $cA$};
\node at (0.4,2) {\small $cB$};
\node at (-2,-0.15) {\small $c$};

\node at (0.2,-0.5) {\small $f$};

\node at (2.25,0) {\small $F$};
\node at (2.25,-1) {\small $Y$};

\node at (0,-1.25) {\small $\pa\sigma$};
\node at (1,-1.2) {\small $\sigma$};
\node at (-2,-1.2) {\small $\beta$};

\end{tikzpicture}
\end{figure}

For any two spaces $A$ and $B$, we glue cones $cA$ and $cB$ to $F$ by identifying the cone points with $c$. The new space $F'=F\cup_c(cA\cup cB)$ is homotopy equivalent to $F$. We further extend $f$ to $f'\colon F'\to Y$ by mapping the cones to the path $\gamma$ in $Y$. The map is ``straightforward'' on $cA$ and is ``twisted'' on $cB$, as illustrated by the picture. Let $\chi(A)=a$ and $\chi(B)=b$. Then we have
\begin{align*}
\chi(f'^{-1}(\sigma))
&=\chi(f^{-1}(\sigma))+a+2b, \\
\chi(f'^{-1}(\pa\sigma))
&=\chi(f^{-1}(\pa\sigma))+a+3b, \\
\chi(f'^{-1}(Y-\mathring{\sigma}))
&=\chi(f^{-1}(Y-\mathring{\sigma}))+b.
\end{align*}
It is therefore possible to choose $a$ and $b$, such that $\chi(f'^{-1}(\sigma))=\chi(\sigma)=1$ and $\chi(f'^{-1}(\pa\sigma))=\chi(\pa\sigma)=1-(-1)^{\dim\sigma}$. By $\chi(F)=\chi(Y)$ mod $n$, this implies that $\chi(f'^{-1}(Y-\mathring{\sigma}))=\chi(Y-\mathring{\sigma})$ mod $n$. 

The basic construction above reduces the problem to the restriction map $f'|\colon f'^{-1}(Y-\mathring{\sigma})\to Y-\mathring{\sigma}$, which still satisfies the Euler characteristic condition in the lemma. This accommodates an inductive argument. We make the stronger inductive assumption that $f'^{-1}(Y-\mathring{\sigma})$ can be extended to $F''$ by glueing cones (identifying cone points with $c$), and $f'|\colon f'^{-1}(Y-\mathring{\sigma})\to Y-\mathring{\sigma}$ can be extended to $f''\colon F''\to Y-\mathring{\sigma}$, such that $\chi(f''^{-1}(\tau))=1$ mod $n$ for every cell $\tau$ of $Y-\mathring{\sigma}$. Then $\hat{F}=F''\cup_c(cA\cup cB)$ is obtained by glueing cones to $F$ (identifying cone points with $c$). Therefore $\hat{F}$ is homotopy equivalent to $F$, and $\hat{f}=f'\cup f''\colon \hat{F}\to Y$ is homotopy equivalent to $f$. Moreover, we still have $\chi(\hat{f}^{-1}(\tau))=\chi(f''^{-1}(\tau))=1$ for every cell $\tau$ of $Y-\mathring{\sigma}$. By applying Lemma \ref{lem1} to the restriction $\hat{f}|\colon \hat{f}^{-1}(\pa\sigma)\to \pa\sigma$, where cells of $\pa\sigma$ are cells of $Y-\mathring{\sigma}$, we get 
\[
\chi(\hat{f}^{-1}(\pa\sigma))
=\chi(\pa\sigma)
=1-(-1)^{\dim\sigma}
=\chi(f'^{-1}(\pa\sigma)).
\]
On the other hand, we have 
\[
\chi(\hat{f}^{-1}(\sigma))-\chi(\hat{f}^{-1}(\pa\sigma))
=\chi(\hat{f}^{-1}(\mathring{\sigma}))
=\chi(f'^{-1}(\mathring{\sigma}))
=\chi(f'^{-1}(\sigma))-\chi(f'^{-1}(\pa\sigma)).
\]
Therefore $\chi(\hat{f}^{-1}(\sigma))=\chi(f'^{-1}(\sigma))=1$.

There are two problems with the induction argument. The first is that $Y-\mathring{\sigma}$ may not be connected. The second is that there may be only one top cell $\beta$.

The case $Y-\mathring{\sigma}$ not connected happens only when $\sigma$ is a $1$-cell, with only one end $v_0$ attached to $Y-\mathring{\sigma}$, and other end $v_1$ being ``free''. In addition to glueing $cA$ and $cB$, we may further glue a cone $cC$, with the map from $cC$ to $Y$ extending all the way to $v_1$. Then we have
\begin{align*}
\chi(f'^{-1}(\sigma))
&=\chi(f^{-1}(\sigma))+a+2b+c, \\
\chi(f'^{-1}(v_0))
&=\chi(f^{-1}(v_0))+a+3b+c, \\
\chi(f'^{-1}(v_1))
&=\chi(f^{-1}(v_1))+c, \\
\chi(f'^{-1}(Y-\mathring{\sigma}))
&=\chi(f^{-1}(Y-\mathring{\sigma}))+b.
\end{align*}
It is then possible to choose $a,b,c$, such that $\chi(f'^{-1}(\sigma))=\chi(f'^{-1}(v_0))=\chi(f'^{-1}(v_1))=1$. The rest of the inductive argument is the same.

\begin{figure}[htp]
\centering
\begin{tikzpicture}[>=latex]

\draw
	(-3,-1) -- (2,-1)
	(-2.5,0) -- (-1.5,0)
	(-0.5,0) -- (2,0)
	(-2,0) to[out=20,in=180] 
	(2,0.5) -- (2,0.7)
	(-2,0) to[out=30,in=180] 
	(2,0.7);

\draw[dotted]
	(0,0) -- (0,0.7)
	(2,0) -- (2,0.7);
	
\draw[->]
	(0,-0.15) -- (0,-0.85);

\fill
	(0,-1) circle (0.05)
	(2,-1) circle (0.05);

\node at (2.25,0.6) {\small $C$};
\node at (0.4,0.3) {\small $cC$};
\node at (-2,-0.15) {\small $c$};

\node at (0.2,-0.5) {\small $f$};

\node at (2.25,0) {\small $F$};

\node at (0,-1.25) {\small $v_0$};
\node at (2,-1.25) {\small $v_1$};
\node at (1,-1.2) {\small $\sigma$};
\node at (-2,-1.2) {\small $Y-\mathring{\sigma}$};

\end{tikzpicture}
\end{figure}

Finally, we consider the case $Y$ has only one top cell $\beta$. In this case, the assumption already says $\chi(f^{-1}(\beta))=1$, and additional cone construction over $\beta$ does not change this fact. What we need to do is to improve $\chi(f^{-1}(\sigma))$ to $1$ for cells $\sigma$ in $\pa\beta$. The problem is then reduced to the restriction map $f|\colon f^{-1}(\pa\beta)\to\pa\beta$. The induction may continue over $\pa\beta$. 
\end{proof}

\section{Non-empty Fixed Point Set}
\label{nonempty}

We prove Theorem \ref{local}, for the case $F_C=f^{-1}(C)$ is not empty for every component $C$ of $Y^G$. The next section deals with the case some $F_C=\emptyset$. 

To construct the pseudo-equivalence extension, we will first homotopically modify $f\colon F\to Y^G$ to a better map described in Lemma \ref{lem2}. Then we construct the extension by inducting on skeleta of $Y^G$. The following justifies the homotopy modification of $f$.

\begin{lemma}\label{homotopy}
Suppose $Y$ is a finite $G$-complex, and $F$ is a finite complex. If a $G$-map $f\colon F\to Y$ can be extended to a pseudo-equivalence (with $F$ being the fixed set), and we change $f,F,Y$ by equivariant homotopy, then the new map can also be extended to a pseudo-equivalence.
\end{lemma}

\begin{proof}
Suppose $g\colon X\to Y$ is a pseudo-equivalence extension of $f\colon F\to Y$. We homotopically change $f,F,Y$ one by one, and argue about the pseudo-equivalence extension of the new map.

First, suppose $f$ is homotopic to $f'\colon F\to Y$. By the equivariant version of the homotopy extension property, the homotopy extends to a $G$-homotopy from $g\colon X\to Y$ to another $G$-map $g'\colon X\to Y$. Then the $G$-map $g'$ extends $f'$, and $g'$ is still a pseudo-equivalence. 

Second, suppose $\phi\colon F\to F'$ is a homotopy equivalence. Then there is a map $f'\colon F'\to Y$, such that $f\colon F\to Y$ is homotopic to $f'\circ \phi\colon F\to F'\to Y$. By the argument above, $f'\circ \phi$ has pseudo-equivalence extension $g'\colon X\to Y$. Then $X'=X\cup_{\phi} F'$ (glueing $F\sub X$ to $F'$ by $\phi$) is a $G$-complex with $F'$ as the fixed set, and the $G$-map $g'\cup f'\colon X'\to Y$ is a pseudo-equivalence extension of $f'$. 

Finally, suppose $\psi\colon Y\to Y'$ is a $G$-homotopy equivalence. Then $\psi\circ f\colon F\to Y'$ is extended to a pseudo-equivalence $\psi\circ g\colon X\to Y'$.
\end{proof}

For the inductive construction (on skeleta of $Y^G$) of pseudo-equivalence extension, we use the following result.

\begin{lemma}\label{lem3}
For any group $G$ of not prime power order. Suppose $K$ is a finite $G$-complex, and $F=K^G$. 
\begin{enumerate}
\item If $\chi(F)=1$ mod $m_G$, then $K$ can be extended to a finite $G$-complex $X$, such that $F=X^G$, $X$ is $(\dim X-1)$-connected, and $H_{\dim X}(X;{\bb Z})$ is a projective ${\bb Z}G$-module.
\item If $\chi(F)=1$ mod $n_G$, then $K$ can be extended to a finite contractible $G$-complex $X$, such that $F=X^G$. 
\end{enumerate}
\end{lemma}

The first statement is \cite[Theorem 2]{oliver}, and Oliver called the $G$-complex $X$ in the statement {\em $G$-resolution}. The second statement is essentially the corollary to \cite[Theorem 3]{oliver}. More precisely, we may use \cite[Theorem 2]{oliver2}, which is even applicable to compact Lie group actions: Let ${\mc F}$ be a {\em family} of subgroups of $G$ as defined by tom Dieck, meaning closed under subgroup and conjugation. If ${\mc F}$ contains all the prime toral subgroups, $X$ and $Z$ are finite $G$-complexes, $Z$ is contractible, and $\chi(X^H/NH)=\chi(Z^H/NH)$ for all $H\not\in {\mc F}$, then $X$ can be extended to a finite contractible $G$-complex $Y$, such that the isotropy subgroups of $Y-X$ are in ${\mc F}$. 

Let
\[
\delta_H=\sum_j(-1)^j(\text{number of cells of type }G/H\times D^j).
\]
Then the proof of \cite[Lemma 14]{oliver} shows that $\chi(X^H/NH)=\chi(Z^H/NH)$ for all $H\not\in {\mc F}$ if and only if $\delta_H(X)=\delta_H(Z)$ for all $H\not\in {\mc F}$.

By \cite[Theorem 3]{oliver2}, the assumption $\chi(F)=1$ mod $n_G$ implies $F=Z^G$ for a finite contractible $G$-complex $Z$. Let ${\mc F}$ be the family of all the prime toral subgroups of $G$. By adding $G/H\times D^j$ to $K$, for $H\not\in{\mc F}\cup\{G\}$, it is easy to get a finite $G$-complex $L$, such that $\delta_H(L)=\delta_H(Z)$ for all $H\not\in {\mc F}\cup\{G\}$. Since $H\ne G$ in the construction of $L$, we have $L^G=K^G=F=Z^G$. Therefore we also have $\delta_G(L)=\chi(L^G)=\chi(Z^G)=\delta_G(Z)$, and we get $\delta_H(L)=\delta_H(Z)$ for all $H\not\in {\mc F}$. This implies $\chi(L^H/NH)=\chi(Z^H/NH)$ for all $H\not\in {\mc F}$. Then by the interpretation of \cite[Theorem 2]{oliver2} above, $L$ extends to a finite contractible $G$-complex $X$, such that the isotropy subgroups of $X-L$ are in ${\mc F}$. Since $G$ is not an isotropy subgroup of $X-L$, we have $X^G=L^G=F$.

\begin{proof}[Proof of Theorem \ref{local} in case all $F_C\ne\emptyset$]
The $G$-complex $Y$ is $G$-homotopic to a regular $G$-complex. If $F_C$ are not empty, then we apply Lemma \ref{lem2} to homotopically modify all $F_C\to C$, such that $\chi(f^{-1}(\sigma))=1$ mod $n_G$ for every cell $\sigma$ of $Y^G$. By Lemma \ref{homotopy}, it is sufficient to construct pseudo-equivalence extension under the additional assumption.

Denote $Z=Y^G$, which has trivial $G$-action. We first extend $f\colon F\to Z$ to a pseudo-equivalence $h\colon W\to Z$ by inducting on the skeleta of $Z$. 

We assume $F^{k-1}=f^{-1}(Z^{k-1})$ is already extended to a $G$-complex $W^{k-1}$ with $(W^{k-1})^G=F^{k-1}$. Moreover, we assume that $f|_{F^{k-1}}$ is extended to a $G$-map $h_{k-1}\colon W^{k-1}\to Z^{k-1}$, such that $h_{k-1}^{-1}(\sigma)$ is contractible for every cell $\sigma$ of $Z^{k-1}$. The inductive assumption holds for $k=0$, because $Z^{-1}=F^{-1}=\emptyset$.

Let $\sigma$ be a $k$-cell of $Z$. Then $\chi(f^{-1}(\sigma))=1$ mod $n_G$ by our assumption. Taking $f^{-1}(\sigma)$ and $f^{-1}(\sigma)\cup h_{k-1}^{-1}(\pa\sigma)$ as $F$ and $K$ in the second part of Lemma \ref{lem3}, we may extend $f^{-1}(\sigma)\cup h_{k-1}^{-1}(\pa\sigma)$ to a finite contractible $G$-complex $W_{\sigma}$, such that $W_{\sigma}^G=f^{-1}(\sigma)$. Since $\sigma$ is contractible and has trivial $G$-action, we may further arrange to extend $f|_{\sigma}\cup h_{k-1}|_{\pa\sigma}$ to a $G$-map $h_{\sigma}\colon W_{\sigma}\to \sigma$, such that $h_{\sigma}^{-1}(\pa\sigma)= h_{k-1}^{-1}(\pa\sigma)$.

Let $W^k=W^{k-1}\cup(\cup_{\dim \sigma=k}W_{\sigma})$, where the union identifies $h_{k-1}^{-1}(\pa\sigma)\sub W_{\sigma}$ with the same subset in $W^{k-1}$. Then we have $G$-map $h_k=h_{k-1}\cup(\cup_{\dim\sigma=k}h_{\sigma})\colon W^k\to Z^k$, such that $(W^k)^G=F^{k-1}\cup (\cup_{\dim \sigma=k}f^{-1}(\sigma))=F_k$. Moreover, we have $h_k^{-1}(\sigma)=W_{\sigma}$ if $\dim\sigma=k$, and $h_k^{-1}(\sigma)=h_{k-1}^{-1}(\sigma)$ if $\dim\sigma<k$. Therefore $h_k^{-1}(\sigma)$ is contractible for every cell $\sigma$ of $Z^k$.

When $k=\dim Z$, we get $h=h_{\dim Z}\colon W=W^{\dim Z}\to Z$, such that $W^G=F$, and $h^{-1}(\sigma)$ is contractible for every cell $\sigma$ of $Z$. This implies that $h\colon W\to Z=Y^G$ is a homotopy equivalence.

Next, we further extend $h\colon W\to Z=Y^G\sub Y$ to a pseudo-equivalence $g\colon X\to Y$.

The equivariant neighborhood $\text{nd}(Z)$ of $Z$ in $Y$ is the mapping cylinder of a $G$-map $\lambda\colon E\to Z$. We try to factor $\lambda$ through a $G$-map $\tilde{\lambda}\colon E\to W$. Then $\lambda=h\circ \tilde{\lambda}$, and we have a $G$-map from the mapping cylinder of $\tilde{\lambda}\colon E\to W$ to the mapping cylinder of $\lambda\colon E\to Z$. The $G$-map extends to a $G$-map $g=id\cup h\colon X=(Y-\text{nd}(Z))\cup_{\tilde{\lambda}}W\to Y=(Y-\text{nd}(Z))\cup_{\lambda}Z$. We have $X^G=W^G=F$, and $g$ extends $f$. Moreover, since $h$ is a pseudo-equivalence, $g$ is also a pseudo-equivalence.

It remains to construct the lifting $\tilde{\lambda}$. Since $G$ fixes no points on $E$, we can construct the lifting if $h\colon W\to Z$ is highly connected for the fixed sets of proper subgroups of $G$ acting on $W$. Recall that we actually constructed $h\colon W\to Z$ in such a way that, for every cell $\sigma$ in the (regular) CW-complex $Z$, $h^{-1}(\sigma)$ is contractible. Let $S$ be the disjoint union of all the $G$-orbits appearing in $E$. Then $S$ is a compact set, such that all the isotropies on $E$ appear in $S$. Then we may take the cell-wise join of $h\colon W\to Z$ with $S\times Z\to Z$ several times to get $h'\colon W'\to Z$. This means $h'^{-1}(\sigma)=h^{-1}(\sigma)*S*\cdots*S$. Since $h^{-1}(\sigma)$ is contractible, $h'^{-1}(\sigma)$ is still contractible. Since $G$ fixes no points of $S$, we get $W'^G=W^G=F$. Therefore $h'$ is still a pseudo-equivalence extension of $f$. On the other hand, the fixed sets of proper subgroups of $G$ acting on $h'^{-1}(\sigma)$ become more and more highly connected as we repeat the join construction more and more times. Therefore we may construct the lifting $\tilde{\lambda}$ by using $h'$ instead of $h$. 
\end{proof}

For the $G$-ANR case, we only need to modify the proof in the very last call. We may use Quinn's ``wrinkling'' trick from \cite[Section 2]{quinn} to remove the $\tilde{K}_0$-obstruction arising in the first part of Lemma \ref{lem3}.

\section{Empty Fixed Point}
\label{empty}

We still need to prove Theorem \ref{local} for the case some $F_C=\emptyset$. In this case, $\chi(F_C)=\chi(C)$ mod $n_G$ means $\chi(C)=0$ mod $n_G$. A typical example is that $C$ is the circle $S^1$, and our proof starts with this special case. In fact, we will concentrate on the case $Y=Y^G=S^1$.

\begin{lemma}\label{circle}
Suppose $G$ is a group of not prime power order. Then there is a pseudo-equivalence $X\to S^1$, such that $X$ has no fixed point, and $G$ fixes $S^1$.
\end{lemma}

The idea is to find a simply connected $G$-space $Z$ without fixed points, and a $G$-map $h\colon Z\to Z$ inducing zero homomorphism on the reduced homology. Then the mapping torus $X$ of $h$ together with the natural map to $S^1$ gives what we want. We may take $Z=S(V)$ and take $h$ to be the self map of $S(V)$ in the following result.

\begin{lemma}\label{rep}
Suppose $G$ is a group of not prime power order, and $V$ is a linear $G$-representation. If $V^G=0$ and all Sylow subgroups of $G$ are isotropy groups of $V$, then there is a degree $0$ $G$-map from the unit sphere $S(V)$ to itself.
\end{lemma}

\begin{proof}
Let $P$ be a Sylow subgroup and let $N_P$ be its normalizer in $G$. Since $P$ is an isotropy subgroup, the fixed subspace $V^P$ is not a zero subspace. Let $D$ be a small equivariant disk neighborhood of a point $x\in S(V^P)$. Then $D$ is an $N_P$-representation. Moreover, since $V$ is a linear representation, the $N_P$-representation is independent of the size of $D$. This means that the radial extension gives an $N_P$-equivariant homeomorphism $D/\pa D\cong S(V)$ sending $*=\pa D/\pa D$ to $x$. Then we may construct an $N_P$-map
\[
S(V)=(S(V)-D)\cup_{\pa D} D 
\to S(V)\vee_x D/\pa D
\to S(V).
\]
The first map collapses $\pa D$ to $x$, and the second map uses the $N_P$-equivariant homeomorphism $D/\pa D\cong S(V)$. The map can be extended to a $G$-map
\[
h_x\colon S(V)
\to S(V)\cup_{Gx}(G\times_{N_P} D/\pa D)
\to S(V).
\]

If we fix an orientation of $S(V)$, then $D$ inherits the orientation, and the homeomorphism $D/\pa D\cong S(V)$ has degree $\pm 1$. By composing with the $-1$ map along a $1$-dimensional subspace of $V^P$, we may change the sign of the degree of the homeomorphism. Therefore we can arrange to have the degree of $h_x$ to be $1+|G/N_P|$ or to be $1-|G/N_P|$. If we apply the construction at several points $x\in S(V^P)$ with disjoint orbits $Gx$, then we get a $G$-map $S(V)\to S(V)$ of degree $1+a|G/N_P|$ for any integer $a$. If we apply the construction to the Sylow subgroups $P_1,P_2,\dots,P_n$ for all the distinct prime factors of $|G|$, then we get a $G$-map $S(V)\to S(V)$ of degree $1+\sum a_i|G/N_{P_i}|$, where $a_1,a_2,\dots,a_n$ can be any prescribed integers. Since $|G/N_{P_1}|,|G/N_{P_2}|,\dots,|G/N_{P_n}|$ are coprime, we get degree $0$ by suitable choice of the integers $a_i$.
\end{proof}

\begin{remark}
We may further make $S(V)$ highly connected for the fixed sets of proper subgroups. Specifically, the kernel of the augmentation $\epsilon(\sum_{g\in G}a_gg)=\sum a_g\colon {\bb R}G\to {\bb R}$ is a representation satisfying the condition of the lemma. The direct sum of several copies of this kernel also satisfies the condition of the lemma. By taking the direct sum of sufficiently many copies, the fixed sets of $S(V)$ for proper subgroups are highly connected. 
\end{remark}

\begin{remark}
Lemma \ref{circle} is valid for non-prime toral compact Lie groups. Moreover, the remark above on the high connectivity is also valid. 

Suppose a compact Lie group $G$ is not prime toral. If the identity component $G_0$ is not abelian (i.e., not torus), then by \cite[Theorem 5]{oliver2}, there is a finite contractible $G$-complex $Z$ without fixed points. In fact, $Z$ can be a disk with smooth $G$-action. Then $X=Z\times S^1\to S^1$ is a pseudo-equivalence, and $X^G=\emptyset$. 

If $G_0$ is abelian, then the order of $G/G_0$ is not prime power. We may apply Lemma \ref{rep} to $G/G_0$, and then take the mapping cylinder to construct $X$. We obtain a $G/G_0$-pseudo-equivalence $X\to S^1$, and $X$ has no $G/G_0$-fixed points. This induces a $G$-pseudo-equivalence $X\to S^1$, and $X$ still has no $G$-fixed points. 
\end{remark}

Alternatively, we may use Bartsch's study of the existence of Borsuk-Ulam theorems \cite{bartsch} to find degree $0$ $G$-map from a fixed point free representation sphere to itself. The equivalence of properties ($c$) and ($d$) of his Theorem 1 gives such a map for finite groups of non-prime power order. The equivalence of properties ($c$) and ($d'$) of his Theorem 2 gives such a map for non-prime toral compact Lie groups. The map on the representation sphere has degree $0$ because it takes the whole sphere into a proper sub-sphere.

\begin{proof}[Proof of Theorem \ref{local} in case some $F_C=\emptyset$]
Assume $F_C=\emptyset$ for some $C$. Then the condition $\chi(F_C)=\chi(C)$ mod $n_G$ means $\chi(C)=0=\chi(S^1)$ mod $n_G$. 

If $F_C=\emptyset$, then we introduce $F'_C=S^1\to C$, where the map can be any one. If $F_C\ne\emptyset$, then we let $F'_C=F_C$, and let the map $F'_C\to C$ be $F_C\to C$. Then $f\colon F\to Y^G$ extends to $f'\colon F'=\cup F'_C\to Y^G$. The modification $f'$ satisfies the Euler characteristic condition in the theorem, and all $F'_C$ are not empty. Since the theorem is already proved for the case all $F_C\ne\emptyset$, $f'$ has a pseudo-equivalence extension $X'\to Y$, with $X'^G=F'$. 

It remains to homotopically replace the extra circles added to $F$ by something that have no fixed points. The equivariant neighborhood $\text{nd}(S^1)$ of one such circle in $X'$ is the mapping cylinder of a $G$-map $\lambda\colon E\to S^1$. By Lemma \ref{circle} and the remarks after the proof of Lemma \ref{rep}, there is a highly connected (for the fixed sets of proper subgroups) pseudo-equivalence $\mu\colon W\to S^1$, such that $W$ has no fixed point. By the high connectivity, $\lambda$ can be lifted to a $G$-map $\tilde{\lambda}\colon E\to W$. Then we have $\lambda=\mu\circ\tilde{\lambda}$. Let $X=(X'-\text{nd}(S^1))\cup_{\tilde{\lambda}}W$ be obtained by glueing the boundary $E$ of $\text{nd}(S^1)$ to $W$, and this is done for all $F'_C=S^1\sub X'$ that were used to replace empty $F_C$. Then $\lambda=\mu\circ\tilde{\lambda}$ and the pseudo-equivalence $\mu$ induce a pseudo-equivalence $X\to X'=(X'-\text{nd}(S^1))\cup_{\lambda}S^1$. The composition $X\to X'\to Y$ is then a pseudo-equivalence with $X^G=F$ and extending $f$.
\end{proof}

We end the section by using Lemma \ref{circle} to prove Theorem \ref{aspherical}.

\begin{proof}[Proof of Theorem \ref{aspherical}]
By Lemma \ref{circle}, there is a $G$-complex $X$ without fixed point, and a pseudo-equivalence $f\colon X\to S^1$, where $G$ fixes $S^1$. By thickening, we may assume $X$ is a manifold with boundary. Let $C$ be the mapping cylinder of $f$. Then the inclusion $X\to C$ is a pseudo-equivalence. We can now do a Davis construction \cite{davis,dh} equivariantly on $X$ (by triangulating the boundary) and mapping to the Davis construction on $C$ (with respect to the boundary of $X$). This produces a $G$-action on a closed aspherical manifold $M$, with a pseudo-equivalence to the same construction on $C$. Since $C$ has a fixed point, the $G$-action on $\pi=\pi_1C=\pi_1X$ lifts to $\text{Aut}(\pi)$. On the other hand, the $G$-action on $M$ has no fixed point, because the original action on $X$ did not.
\end{proof}

It is an interesting problem whether the action constructed in the proof can exist on classical aspherical manifolds, e.g., hyperbolic manifolds. For odd order cyclic group (or the order is a power of $2$), \cite{we} shows the answer is negative.

\section{Obstruction Group}
\label{obstruction}

We prove Theorem \ref{obstruct}.

For a $G$-complex $Y$, let $N_Y\sub {\bb Z}^{\pi_0Y^G}$ be the collection of 
\[
\nu(g)
=(\chi(F_C)-\chi(C))_{C\in \pi_0Y^G},\quad
F_C=g^{-1}(C)\cap X^G,
\]
for all pseudo-equivalences $g\colon X\to Y$. We prove $N_Y$ is a group, by showing that it is closed under negative and addition operations.

For a pseudo-equivalence $g\colon X\to Y$, we construct its {\em negative} 
\[
\bar{g}\colon \bar{X}=Y\cup X\times[0,1]\cup Y\to Y
\]
to be the double mapping cylinder of $g$. Then $\bar{g}$ is still a pseudo-equivalence, with $\bar{F}_C=C\cup F_C\times[0,1]\cup C$, and $\chi(\bar{F}_C)-\chi(C)=-(\chi(F_C)-\chi(C))$. Therefore $\nu\in N_Y$ implies $-\nu\in N_Y$.

The negative construction has the following properties:
\begin{enumerate}
\item $\bar{X}$ contains a copy of $Y$, and the non-equivariant homotopy equivalence can be a homotopy retraction of $\bar{X}$ to $Y$.
\item $\bar{F}_C$ is connected. Therefore the connected components of the fixed sets of $\bar{X}$ and $Y$ are in one-to-one correspondence.
\end{enumerate}
We call a pseudo-equivalence with the two properties {\em retracting equivalence}. Since the double negative satisfies $\nu(\bar{\bar{g}})=\nu(g)$, every element in $N_Y$ is represented by a retracting equivalence. 

For two retracting equivalences $g_1\colon X_1\to Y$ and $g_2\colon X_2\to Y$, the {\em addition} $g_1\cup g_2\colon X_1\cup_Y X_2\to Y$ is still a retracting equivalence. It is also easy to see that $\nu(g_1\cup g_2)=\nu(g_1)+\nu(g_2)$. Therefore $\nu_1,\nu_2\in N_Y$ implies $\nu_1+\nu_2\in N_Y$.

This completes the proof that $N_Y$ is an abelian subgroup. Next, we prove that $N_Y$ is indeed the obstruction to pseudo-equivalence extension. 

By the definition of $N_Y$, if $f\colon F\to Y$ extends to a pseudo-equivalence $g\colon X\to Y$, such that $X^G=F$, then $(\chi(F_C)-\chi(C))_{C\in \pi_0Y^G}=\nu(g)\in N_Y$.

Conversely, suppose $f\colon F\to Y$ satisfies $(\chi(F_C)-\chi(C))_{C\in \pi_0Y^G}\in N_Y$. Then $(\chi(F_C)-\chi(C))_{C\in \pi_0Y^G}=\nu(g')$ for some a pseudo-equivalence $g'\colon X'\to Y$. This means $\chi(F_C)=\chi(F'_C)$, where $F'_C=g'^{-1}(Y^G_C)\cap X'^G$. 

As remarked earlier, we may further assume that $g'$ is a retracting pseudo-equivalence. Then we may regard $f$ as mapped into $Y\sub X'$. This means $f$ is a composition ($i\colon Y\to X'$ is the inclusion)
\[
f=g'\circ(i\circ f)\colon
F\xrightarrow{i\circ f}X'
\xrightarrow{g'} Y.
\]
Since $Y\sub X'$, and connected components of the fixed sets of $X'$ and $Y$ are in one-to-one correspondence, we have 
\[
C=Y\cap F'_C,\quad
F_C=f^{-1}(Y\cap F'_C)=(i\circ f)^{-1}(F'_C),
\]
and 
\[
\chi((i\circ f)^{-1}(F'_C))
=\chi(F_C)
=\chi(F'_C).
\]
By Theorem \ref{local}, this implies that $i\circ f$ has pseudo-equivalence extension $h\colon X\to X'$. Then $g'\circ h\colon X\to Y$ is a pseudo-equivalence extension of $f$. This completes the proof of Theorem \ref{obstruct}.

\section{The Role of Fundamental Group}
\label{fundgroup}

In this section, we develop the equivariant Euler-Wall characteristic of a $G$-complex that lie in $K_0(R[\Gamma])$ for rings $R$ in which the orders of isotropy groups are invertible. We apply this to get further restrictions on the Euler characteristics of components of fixed sets under pseudo-equivalences.

We first elaborate on the lifted $G$-actions to the universal cover that we use to define the weakly $G$-connected property in the introduction.

Let $p\colon \tilde{Y}\to Y$ be the universal cover, with free action by the fundamental group $\pi=\pi_1Y$. A $G$-action on $Y$ lifts to self homeomorphisms of $\tilde{Y}$. All the liftings form a group $\Gamma$ fitting into an exact sequence
\begin{equation}\label{eq1}
1\to \pi\to \Gamma\to G\to 1,
\end{equation}

Let $\tilde{y}\in \tilde{Y}$, $y=p(\tilde{y})$. The induced homomorphism $\Gamma_{\tilde{y}}\to G_y$ of isotropy groups is an isomorphism. If $y\in Y^G$, then we get a splitting $G=G_y\cong \Gamma_{\tilde{y}}\sub \Gamma$ of \eqref{eq1}. If $\tilde{y}$ and $\tilde{y}'$ are in the same connected component $\hat{C}$ of $p^{-1}(Y^G)$, then $\Gamma_{\tilde{y}}=\Gamma_{\tilde{y}'}$. Therefore we may denote $\Gamma_{\tilde{y}}=\Gamma_{\hat{C}}$, and the splitting $G\cong \Gamma_{\hat{C}}\sub \Gamma$ depends only on $\hat{C}$.

The component $\hat{C}$ covers a connected component $C$ of $Y^G$. The other connected components of $p^{-1}(C)$ are $a\hat{C}$, $a\in \pi$. Therefore a connected component $C$ of $Y^G$ gives a {\em $\pi$-conjugacy class of splittings} of \eqref{eq1}
\[
\Gamma_C
=\{\Gamma_{a\hat{C}}=a\Gamma_{\hat{C}}a^{-1}\colon a\in \pi\}.
\]

\begin{example}
The complex conjugation action of $G={\bb Z}_2$ on circle $Y=S^1$ has fixed point components $C_1=\{1\}$ and $C_{-1}=\{-1\}$. The universal cover is $p(t)=e^{it}\colon \tilde{Y}={\bb R}\to Y$. The group $\Gamma$ consists of $\sigma_n(t)=t+2n\pi$ (liftings of the identity, which form $\pi_1Y$) and $\rho_n(t)=2n\pi-t$ (liftings of the conjugation). We have 
\begin{align*}
p^{-1}(C_1)
&=\{2n\pi\colon n\in {\bb Z}\}, &
p^{-1}(C_{-1})
&=\{(2n+1)\pi\colon n\in {\bb Z}\}; \\
\hat{C}_1
&=\{0\}, &
\hat{C}_{-1}
&=\{\pi\}; \\
\Gamma_{\{2n\pi\}}
&=\{1,\rho_{2n}\}
=\sigma_1^n\Gamma_{\hat{C}_1}\sigma_1^{-n}, &
\Gamma_{\{(2n+1)\pi\}}
&=\{1,\rho_{2n+1}\}
=\sigma_1^n\Gamma_{\hat{C}_{-1}}\sigma_1^{-n}.
\end{align*}
We have two conjugate families of splittings 
\[
\Gamma_{C_1}=\{\Gamma_{\{2n\pi\}}\},\quad
\Gamma_{C_{-1}}=\{\Gamma_{\{(2n+1)\pi\}}\}.
\]
\end{example}

A splitting of \eqref{eq1} corresponds to a semi-direct product decomposition $\Gamma=\pi\rtimes G$. The $\pi$-conjugacy classes of splittings form the cohomology set $H^1(G;\pi)$ (not necessarily a group because $\pi$ may not be commutative). 

A $G$-map $g\colon X\to Y$ has a pullback $\tilde{g}\colon \tilde{X}\to \tilde{Y}$ along the universal cover $p\colon \tilde{Y}\to Y$. The map $\tilde{g}$ is a $\Gamma$-map, and induces a map of ${\bb Z}[\Gamma]$-chain complexes $\tilde{g}_*\colon C(\tilde{X})\to C(\tilde{Y})$. If $g$ is a pseudo-equivalence, then $\tilde{g}_*$ has a ${\bb Z}[\pi]$-chain homotopy inverse $\varphi$. 

Let $R$ be a ring, say the rational numbers ${\bb Q}$, such that $|G|$ is invertible in $R$. Then we may use one splitting $\Gamma=\pi\rtimes G$ to get a $R[\Gamma]$-chain map $\frac{1}{|G|}\sum_{u\in G}u\varphi\colon C(\tilde{Y})\to C(\tilde{X})$. This is a $R[\Gamma]$-chain homotopy inverse of 
\[
\tilde{g}_*\otimes R\colon C(\tilde{X})\otimes R\to C(\tilde{Y})\otimes R.
\]
In particular, $\tilde{g}_*\otimes R$ is a $R[\Gamma]$-chain homotopy equivalence.

Since $|G|$ is invertible in $R$, and the isotropy groups of the $\Gamma$-action are isomorphic to subgroups of $G$, we know $C(\tilde{X})\otimes R$ and $C(\tilde{Y})\otimes R$ consist of finitely generated projective $R[\Gamma]$-modules. Then the $R[\Gamma]$-chain complexes give the Euler characteristic elements $\chi_{\Gamma}(\tilde{X})$ and $\chi_{\Gamma}(\tilde{Y})$ in $K_0(R[\Gamma])$. The $R[\Gamma]$-chain homotopy equivalence $\tilde{g}_*\otimes R$ implies $\chi_{\Gamma}(\tilde{X})=\chi_{\Gamma}(\tilde{Y})$.

The $G$-cells $G\sigma$ of $Y$ are in one-to-one correspondence with $\Gamma$-cells $\Gamma\tilde{\sigma}$ of $\tilde{Y}$, where $\tilde{\sigma}$ is any cell of $\tilde{Y}$ over $\sigma$. The Euler characteristic of $C(\tilde{Y})\otimes R$ is 
\[
\chi_{\Gamma}(\tilde{Y})
=\sum_{\text{$G$-cells of $Y$}}(-1)^{\dim\sigma}[R[\Gamma\tilde{\sigma}]].
\]
Here $\Gamma\tilde{\sigma}=\Gamma/\Gamma_{\tilde{\sigma}}$ is a $\Gamma$-orbit, and $R[\Gamma\tilde{\sigma}]$ is a projective $R[\Gamma]$-module. For finite subgroup $H$, we know the {\em rank} of the projective $R[\Gamma]$-module $R[\Gamma/H]$ is ($(\gamma)$ is the conjugacy class of $\gamma$ in $\Gamma$)
\[
\text{rank}(R[\Gamma/H])
=\frac{1}{|H|}\sum_{h\in H}h
\in \oplus_{(\gamma)\sub \Gamma}R(\gamma).
\]
By $\chi_{\Gamma}(\tilde{X})=\chi_{\Gamma}(\tilde{Y})$, we have $\text{rank}(\chi_{\Gamma}(\tilde{X}))=\text{rank}(\chi_{\Gamma}(\tilde{Y}))$.

\begin{theorem}\label{cyclic}
Suppose $G$ is a cyclic group acting on a finite $G$-complex $Y$. Suppose $\gamma\in \Gamma$ is mapped to a generator of $G$, and $\langle\gamma\rangle$ is the cyclic subgroup generated by $\gamma$. If $g\colon X\to Y$ is a rational pseudo-equivalence, then 
\[
\sum_{\langle\gamma\rangle\in \Gamma_C}\chi(F_C)
=\sum_{\langle\gamma\rangle\in \Gamma_C}\chi(C).
\] 
The sum is over all components $C$ of $Y^G$ satisfying $\langle\gamma\rangle\in \Gamma_C$. 
\end{theorem}

If $Y$ is weakly $G$-connected, then there is at most one $C$ satisfying $\langle\gamma\rangle\in \Gamma_C$, and the proposition says $\chi(F_C)=\chi(C)$ for each $C$. In other words, the component wise Euler characteristic condition in Theorem \ref{local} is necessary and sufficient. Therefore we have $N_Y=n_G{\bb Z}^{\pi_0Y^G}$.

\begin{proof}
Since $\gamma\in \Gamma$ is mapped to a generator of $G$, by the formula for $\text{rank}({\bb Q}[\Gamma\tilde{\sigma}])=\text{rank}({\bb Q}[\Gamma/\Gamma_{\tilde{\sigma}}])$, the conjugacy class $(\gamma)$ appears in $\text{rank}({\bb Q}[\Gamma\tilde{\sigma}])$ if and only if a conjugation of $\gamma$ fixes $\tilde{\sigma}$. Moreover, for such $\tilde{\sigma}$, we have
\[
\text{rank}({\bb Q}[\Gamma\tilde{\sigma}])
=\frac{1}{n}\sum_{i=0}^{n-1}(\gamma^i),\quad
n=|G|.
\]
The elements $\gamma^i$ are not $\Gamma$-conjugate because they are mapped to non-conjugate elements of the cyclic group $G$. Therefore the coefficient of $(\gamma)$ in $\text{rank}({\bb Q}[\Gamma\tilde{\sigma}])$ is $\frac{1}{n}$. 

The terms ${\bb Q}[\Gamma\tilde{\sigma}]$ of $C(\tilde{Y})\otimes {\bb Q}$ are in one-to-one correspondence with $G$-cells $G\sigma$ of $Y$. If a conjugation of $\gamma$ fixes $\tilde{\sigma}$, by $\gamma$ mapped to the generator of $G$, we know $G$ fixes $\sigma$. Therefore $\sigma$ is in a component $C$ of $Y^G$ satisfying $\langle\gamma\rangle\in \Gamma_C$. Conversely, any such $\sigma$ gives $\frac{1}{n}(\gamma)$ in the corresponding $\text{rank}({\bb Q}[\Gamma\tilde{\sigma}])$. Therefore the coefficient of $(\gamma)$ in 
\[
\text{rank}(\chi_{\Gamma}(\tilde{Y}))
=\sum_{\text{$G$-cells $G\sigma$ of $Y$}} (-1)^{\dim\sigma}\text{rank}({\bb Q}[\Gamma\tilde{\sigma}])
\]
is 
\[
\sum_{\text{$G$-cell $G\sigma$ of $C$ satisfying $\langle\gamma\rangle\in \Gamma_C$}}(-1)^{\dim\sigma}\frac{1}{n}
=\frac{1}{n}\sum_{\langle\gamma\rangle\in \Gamma_C}\chi(C).
\]
We have the same calculation for the pullback $\tilde{X}$, and find that the coefficient of $(\gamma)$ in $\text{rank}(\chi_{\Gamma}(\tilde{X}))$ is $\frac{1}{n}\sum_{\langle\gamma\rangle\in \Gamma_C}\chi(F_C)$. Then we conclude $\frac{1}{n}\sum_{\langle\gamma\rangle\in \Gamma_C}\chi(F_C)=\frac{1}{n}\sum_{\langle\gamma\rangle\in \Gamma_C}\chi(C)$.
\end{proof}

Next, we apply the idea to $G$ satisfying $n_G=0$. This means there is a normal subgroup $P$, such that $|P|=p^l$ for a prime $p$, and $G/P$ is cyclic of order $n$. We may further assume that $p$ and $n$ are coprime. The liftings of $P$-actions give an exact sequence
\begin{equation}\label{eq2}
1\to\pi\to \Pi\to P\to 1.
\end{equation}
Here $\Pi$ is the pre-image of $P$ under $\Gamma\to G$, and \eqref{eq2} is part of \eqref{eq1}.

Each connected component $C$ of $Y^G$ is contained in a connected component $D$ of $Y^P$. Then $p^{-1}(C)\sub p^{-1}(D)$, and each connected component $\hat{C}$ of $p^{-1}(C)$ is contained in a connected component $\hat{D}$ of $p^{-1}(D)$. The pair $(\hat{C},\hat{D})$ gives a pair of compatible splittings $G\cong \Gamma_{\hat{C}}\sub G$ and $P\cong \Gamma_{\hat{D}}\sub \Pi$ of \eqref{eq1} and \eqref{eq2}. All the pairs $(\hat{C},\hat{D})$ are related by $\pi$-translations, and the corresponding pairs of splittings are $\pi$-conjugate. Then we get a $\pi$-conjugacy class of compatible splittings (we fix one pair $(\hat{C},\hat{D})$ in the second expression)
\[
\Gamma_{CD}
=\{\text{all }(\Gamma_{\hat{C}},\Gamma_{\hat{D}})\}
=\{(a\Gamma_{\hat{C}}a^{-1},a\Gamma_{\hat{D}}a^{-1})\colon a\in \pi\}.
\]

\begin{theorem}\label{comp-wise}
Suppose $P$ is a normal $p$-subgroup of $G$, and $G/P$ is a cyclic group of order prime to $p$. Suppose $X,Y$ are finite $G$-complexes, and a $G$-map $g\colon X\to Y$ is a pseudo-equivalence. Then for connected components $C_0,D_0$ of $Y^G,Y^P$ satisfying $C_0\sub D_0$, we have
\[
\sum_{\Gamma_{CD_0}=\Gamma_{C_0D_0}}\chi(F_C)
=\sum_{\Gamma_{CD_0}=\Gamma_{C_0D_0}}\chi(C).
\]
The sum is over all components $C$ of $Y^G$ satisfying $\Gamma_{CD_0}=\Gamma_{C_0D_0}$. 
\end{theorem}

The condition $\Gamma_{CD_0}=\Gamma_{C_0D_0}$ means the following. We fix one component $D$ (denoted $D_0$ in the proposition) of $Y^P$, and consider all the components $C$ of $Y^G$ that are contained in $D$. Then we further distinguish these $C$ by the conjugation classes of the associated splittings. The sum of Euler characteristics is over these conjugation class. 

If all components $C$ inside $D$ have non-conjugate splittings, then the sum is over single $C$, and we get $\chi(F_C)=\chi(C)$ for every $C$ inside $D$. Furthermore, suppose $Y$ has the property that, if components $C$ and $C'$ of $Y^G$ give conjugate splittings, then $C$ and $C'$ belong to different components of $Y^P$. Of course a weakly $G$-connected $Y$ has this property. Under this property, we get $\chi(F_C)=\chi(C)$ for every component $C$ of $Y^G$. 

\begin{proof}
The cyclic group $H=G/P$ acts on $Y^P$, and the group $\tilde{H}$ of liftings of $H$-actions to $p^{-1}(Y^P)$ fits into an exact sequence
\begin{equation}\label{eq3}
1\to\pi\to \tilde{H}\to H\to 1.
\end{equation}
By Smith theory \cite{cwy-jones,jones}, we know $g^P\colon X^P\to Y^P$ is an ${\bb F}_p[\tilde{H}]$-homology equivalence. This implies $g^P$ is a ${\bb Z}_{p^k}[\tilde{H}]$-homology equivalence for all $k$. 

The homology equivalence is a sum of homology equivalences on connected components \cite{cwy-jones}. Let $C,D,\hat{C},\hat{D}$ be given as in the discussion before the proposition. Since $G$ has fixed set $C$ in $D$, we know $H=G/P$ acts on $D$. Then $\hat{D}$ covers $D$, and the group $H_{\hat{D}}$ of liftings of $H$-actions to $\hat{D}$ fits into an exact sequence
\begin{equation}\label{eq4}
1\to\pi_{\hat{D}}\to H_{\hat{D}}\to H\to 1.
\end{equation}
Here $\pi_{\hat{D}}$ is the subgroup of translations $a\in \pi$ satisfying $a\hat{D}=\hat{D}$, and \eqref{eq4} is part of \eqref{eq3}. Let $\hat{F}_D$ be the pullback of $g^{-1}(D)\cap X^P\to D\leftarrow \hat{D}$, then the restriction of $g^P$ induces a ${\bb Z}_{p^k}[H_{\hat{D}}]$-chain homology equivalence 
\begin{equation}\label{eq5}
g^P_*\colon
C(\hat{F}_D)\otimes{\bb Z}_{p^k} \to 
C(\hat{D})\otimes{\bb Z}_{p^k}.
\end{equation}

We note that $\Gamma_{\hat{C}}/\Gamma_{\hat{D}}\cong H$. In fact, by $\Gamma_{\hat{C}}/\Gamma_{\hat{D}}\sub H_{\hat{D}}$, we have a splitting of \eqref{eq4}. For fixed $D,\hat{D}$, the other choices of $\hat{C}\sub\hat{D}$ over the same $C$ give $\pi_{\hat{D}}$-conjugations of the splitting. These conjugations are in one-to-one correspondence with the conjugations of the pair $(\Gamma_{\hat{C}},\Gamma_{\hat{D}})$. Therefore $\Gamma_{CD}$ is also the $\pi_{\hat{D}}$-conjugacy class of the splittings of \eqref{eq4}.

Let $\gamma\in \Gamma_{\hat{C}}/\Gamma_{\hat{D}}$ correspond to a generator of the cyclic group $H$. Since $p$ and $n$ are coprime, the order $n=|H|=|\Gamma_{\hat{C}}/\Gamma_{\hat{D}}|$ is invertible in the ring $R={\bb Z}_{p^k}$. Therefore both chain complexes in \eqref{eq5} consist of projective ${\bb Z}_{p^k}[H_{\hat{D}}]$-modules, and the homology equivalence is a ${\bb Z}_{p^k}[H_{\hat{D}}]$-chain homotopy equivalence. Then we may apply the same idea in the proof of Theorem \ref{cyclic}, with ${\bb Q}$ replaced by ${\bb Z}_{p^k}$, and get the similar conclusion. We fix $C_0,D_0,\hat{C}_0,\hat{D}_0$, and get the generator $\gamma\in \Gamma_{\hat{C}_0}/\Gamma_{\hat{D}_0}\sub H_{\hat{D}_0}$. Using the conjugacy class $\Gamma_{CD_0}$ explained above, we conclude
\[
\sum_{\langle\gamma\rangle\in \Gamma_{CD_0}}\chi(F_C)
=\sum_{\langle\gamma\rangle\in \Gamma_{CD_0}}\chi(C)\quad \text{mod $p^k$}.
\]
Here the equality is mod $p^k$ because it is an equality in ${\bb Z}_{p^k}$. We note that $\langle\gamma\rangle\in \Gamma_{CD_0}$ is the same as $\Gamma_{CD_0}=\Gamma_{C_0D_0}$. Moreover, the equality holds mod $p^k$ for all $k$ means the equality holds as integers.
\end{proof}


\medskip

\begin{thebibliography}{1}

\bibitem{bartsch}
T.~Bartsch.
\newblock On the existence of Borsuk-Ulam theorems.
\newblock {\em Topology}, 31(3):533-543, 1992.

\bibitem{bass}
H.~Bass.
\newblock Euler characteristics and characters of discrete groups.
\newblock {\em Invent. Math.}, 35:155-196, 1976.

\bibitem{bh}
M.~Bridson., A.~Haefliger.
\newblock Metric Spaces of Non-Positive Curvature.
\newblock Springer-Verlag, Berlin, Heidelberg, 1999.

\bibitem{cwy-jones}
S.~Cappell, S.~Weinberger, M.~Yan.
\newblock Fixed points of semi-free $G$-CW-complex of given homotopy type.
\newblock {\em preprint}, 2021.

\bibitem{davis}
M.~Davis.
\newblock Groups generated by reflections and aspherical manifolds not covered by Euclidean space.
\newblock {\em Ann. of Math.}, 117(2):293-324, 1983.


\bibitem{dh}
M.~Davis, J.~C.~Hausmann.
\newblock Aspherical manifolds without smooth or PL structure.
\newblock LNM 1370, 135-142, 1986.


\bibitem{hattori}
A.~Hattori.
\newblock Rank element of a projective module.
\newblock {\em Nagoya J. Math.}, 25:113-120, 1965.

\bibitem{jones}
L.~Jones.
\newblock The converse to the fixed point theorem of P.A. Smith: I.
\newblock {\em Ann. of Math.}, 94(1):52-68, 1971.

\bibitem{mi}
M.~Morimoto, K.~Iizuka.
\newblock Extendability of $G$-maps to pseudo-equivalences to finite $G$-CW-complexes whose fundamental groups are finite.
\newblock {\em Osaka J. Math.}, 21:59-69, 1984.

\bibitem{oliver}
R.~Oliver.
\newblock Fixed-point sets of group actions on finite cyclic complexes.
\newblock {\em Comment. Math. Helvetici}, 50:155-177, 1975.

\bibitem{oliver2}
R.~Oliver.
\newblock Smooth compact Lie group actions on disks.
\newblock {\em Math. Z.}, 149:79-96, 1976.

\bibitem{oliver3}
R.~Oliver.
\newblock $G$-actions on disks and permutation representations.
\newblock {\em J. Algebra}, 50:44-62, 1978.

\bibitem{oliver4}
R.~Oliver.
\newblock $G$-actions on disks and permutation representations II.
\newblock {\em Math. Z.}, 157:237-263, 1977.

\bibitem{op}
R.~Oliver, T.~Petrie.
\newblock $G$-CW-surgery and $K_0({\bb Z}G)$.
\newblock {\em Math. Z.}, 179:11-42, 1982.

\bibitem{quinn}
F. Quinn.
\newblock Ends of maps II.
\newblock {\em Invent. Math.}, 68(3):353-424, 1982.


\bibitem{smith}
P. A.~Smith.
\newblock Fixed-Point theorems for periodic transformations.
\newblock {\em Amer. J. Math.}, 63(1):1-8, 1941.

\bibitem{stallings}
J.~Stallings.
\newblock Centerless groups - an algebraic formulation of Gottlieb's theorem.
\newblock {\em Topology}, 4:129-134, 1965.

\bibitem{we}
S.~Weinberger.
\newblock A fixed point theorem for periodic maps on locally symmetric manifolds. 
\newblock {\em Algebra i Analiz}, 29(1):60-69, 2017. reprinted in {\em St. Petersburg Math. J.}, 29(1):43-50, 2018.

\bibitem{west}
J. West.
\newblock 
Mapping Hilbert cube manifolds to ANR's: a solution of a conjecture of Borsuk.
\newblock {\em Ann. of Math.}, 106(2):1-18, 1977.


\end{thebibliography}
\end{document}